\documentclass[11pt]{amsart}
\usepackage{amsmath, amssymb, amsthm, amsfonts, mathrsfs}

\newtheorem{theorem}{Theorem}[section]

\newtheorem{corollary}[theorem]{Corollary}

\newtheorem{lemma}{Lemma}[section]

\theoremstyle{definition}

\newcommand{\dom}{{\mbox{\rm dom}}}
\newcommand{\lh}{{\mbox{\rm lh}}}
\newcommand{\id}{{\mbox{\rm id}}}

\DeclareSymbolFont{AMSb}{U}{msb}{m}{n}
\DeclareMathSymbol{\N}{\mathbin}{AMSb}{"4E}
\DeclareMathSymbol{\Z}{\mathbin}{AMSb}{"5A}
\DeclareMathSymbol{\R}{\mathbin}{AMSb}{"52}
\DeclareMathSymbol{\Q}{\mathbin}{AMSb}{"51}
\DeclareMathSymbol{\I}{\mathbin}{AMSb}{"49}
\DeclareMathSymbol{\C}{\mathbin}{AMSb}{"43}

\theoremstyle{definition}

\begin{document}
\title{Disjointness between Bounded Rank-One Transformations}
\author{Su Gao}
\address{Department of Mathematics\\ University of North Texas\\ 1155 Union Circle \#311430\\  Denton, TX 76203\\ USA}
\email{sgao@unt.edu}
\author{Aaron Hill}
\address{Department of Mathematics\\ University of Louisville\\ Louisville, KY 40292\\ USA}
\email{aaron.hill@louisville.edu}
\date{\today}

\begin{abstract} In this paper some sufficient conditions are given for when two bounded rank-one transformations are isomorphic or disjoint. For commensurate, canonically bounded rank-one transformations, isomorphism and disjointness are completely determined by simple conditions in terms of their cutting and spacer parameters. We also obtain sufficient conditions for bounded rank-one transformations to have minimal self-joinings. As an application, we give a proof of Ryzhikov's theorem that totally ergodic, non-rigid, bounded rank-one transformations have minimal self-joinings of all orders.
\end{abstract}
\maketitle\thispagestyle{empty}

\section{Introduction}
The research in this paper is motivated by the observation of Foreman, Rudolph, and Weiss \cite{FRW}, based on King's Weak Closure Theorem for rank-one transformations \cite{King1}, that the isomorphism problem for rank-one transformations is a Borel equivalence relation. Our objective has been to identify a concrete algorithm to determine when two rank-one transformations are isomorphic.

The broader context of this research is the isomorphism problem in ergodic theory, originally posed by von Neumann, that asks how to determine when two (invertible) measure-preserving transformations are isomorphic. 

Recall that a {\em measure-preserving transformation} is an automorphism
of a standard Lebesgue space. Formally, it is a quadruple $(X, \mathcal{B}, \mu, T)$, where
$(X, \mathcal{B}, \mu)$ is a measure space isomorphic to the unit interval with the
Lebesgue measure on all Borel sets, and $T$ is a bijection from $X$ to $X$ such that $T$ and $T^{-1}$ are both $\mu$-measurable and preserve the measure $\mu$.
When the algebra of measurable sets is clear, we refer to the transformation $(X, \mathcal{B}, \mu,T)$ simply by $(X, \mu, T)$. 

Two measure-preserving transformations $(X, \mathcal{B}, \mu, T)$ and $(Y, \mathcal{C}, \nu, S)$ are {\em isomorphic} if there is a measure isomorphism $\varphi$  from $(X, \mathcal{B}, \mu)$ to $(Y, \mathcal{C}, \nu)$ such that $\varphi\circ T=S\circ \varphi$ a.e.

Halmos and von Neumann showed that two ergodic measure-preserving
transformations with pure point spectrum are isomorphic if and only if they have
the same spectrum. Ornstein's celebrated theorem states that two Bernoulli shifts are isomorphic if and only if they have the same entropy. These are successful answers to the isomorphism problem for subclasses of measure-preserving transformations. For each of them, there is a concrete algorithm, which can be carried out at least in theory, to determine when two given measure-preserving transformations are isomorphic.

Foreman, Rudolph, and Weiss \cite{FRW} showed that the isomorphism problem for ergodic
measure-preserving transformations is a complete analytic equivalence relation, and in particular not Borel. Intuitively,
this rules out the existence of a satisfactory answer to the original isomorphism problem of von Neumann. However, in the same paper they showed that the isomorphism relation becomes much simpler when restricted to the generic
class of rank-one transformations. Although their method does not yield a concrete algorithm for the isomorphism problem for rank-one transformations, it gives hope that the isomorphism problem has a satisfactory solution for a generic class of measure-preserving transformations. Since rank-one transformations are given by their cutting and spacer parameters $(r_n:n\in\N)$ and $(s_n: n\in\N)$ (more details are given in the next section), a satisfactory solution to the isomorphism problem would correspond to a simple algorithm that yields a yes or no answer with these parameters as input.

In this paper we make some progress toward such a satisfactory solution. Under the assumption that the cutting and spacer parameters are bounded, we investigate the isomorphism problem and yield some conditions to guarantee isomorphism and non-isomorphism. For the class of canonically bounded rank-one transformations, we are able to give a simple algorithm to determine isomorphism when the rank-one transformations are commensurate or just eventually commensurate (Corollary~\ref{CBmain} and Theorem~\ref{mainisogen}). The basic techniques of this investigation come from the recent \cite{Hill} by the second author.  

In addition to the isomorphism problem, we investigate in this paper a stronger notion of non-isomorphism, namely, that of disjointness between measure-preserving transformations. Two measure-preserving transformations $(X, \mu, T)$ and $(Y, \nu, S)$ are {\em disjoint} if $\mu\times\nu$ is the only measure on $X\times Y$ that is $T\times S$-invariant and has $\mu$ and $\nu$ as marginals. A main result of this paper (Theorem~\ref{maindisjoint}) gives a condition for when two bounded rank-one transformations are disjoint. For the class of canonically bounded rank-one transformations, we again yield a simple algorithm to determine when two commensurate transformations are disjoint (Corollary~\ref{CBmain}).

Our results on isomorphism and disjointness for canonically bounded rank-one transformations extend what was already known for a class of Chacon-like transformations.  Chacon's transformation is a prototypical example of canonically bounded rank-one transformations; it can be described by the cutting parameter that is constantly equal to 3 and the spacer parameter that is constantly equal to (0,1)--i.e., there are no spacers inserted at the first opportunity and a single spacer inserted at the second opportunity.  Given any sequence $e = (e_n : n \in \N)$ of 0s and 1s, we can build a Chacon-like transformation $T_e$ as follows.  The cutting parameter for the transformation will be constantly equal to 3 and the spacer parameter at stage $n$ will be $(0,1)$ if $e_n =1$ and $(1,0)$ if $e_n = 0$.  Fieldsteel \cite{Fieldsteel} showed that transformations $T_e$ and $T_{e^\prime}$ that are constructed in this way are isomorphic iff $e$ and $e^\prime$ eventually agree, i.e., there is some $N \in \N$ such that $e_n = e^\prime_n$ for all $n \geq N$.  It is an exercise in Rudolph's book \cite{RudolphBook} to show that in the case that $e$ and $e^\prime$ do not eventually agree, then $T_e$ and $T_{e^\prime}$ are in fact disjoint.  

The notion of canonically bounded rank-one transformations was defined in \cite{GaoHill} and was used in \cite{GaoHill1} to characterize non-rigidity for bounded rank-one transformations.  In our study of disjointness of commensurate, canonically bounded rank-one transformations, the basic method follows that of del Junco, Rahe, and Swanson \cite{delJuncoRaheSwanson} in which they showed that Chacon's transformation--in fact, any Chacon-like transformation--has minimal self-joinings of all orders.     Further exploration of the method gives us a generalization of the theorem of del Junco, Rahe, and Swanson (Theorem~\ref{mainMSJ}), which gives a general condition when bounded rank-one transformations have minimal self-joinings of all orders. Applying this general condition to canonically bounded rank-one transformations, and combining the results of \cite{GaoHill1}, we are able to obtain a proof of Ryzhikov's theorem \cite{Ryzhikov} that totally ergodic, non-rigid, bounded rank-one transformations have minimal self-joinings of all orders.

The rest of the paper is organized as follows. In Section 2 we provide further background and define the basic notions used throughout the paper. In Section 3 we state and prove the main results on non-isomorphism (Theorem~\ref{mainiso}) and disjointness (Theorem~\ref{maindisjoint}), and derive a satisfactory solution to the isomorphism and disjointness problems for commensurate, canonically bounded rank-one transformations (Corollary~\ref{CBmain}). In Section 4 we state the prove the main result on minimal self-joinings (Theorems~\ref{mainMSJ} and \ref{allorders}), and derive Ryzhikov's theorem (Corollary~\ref{Ryzh}) from our methods. In the final section, we give some concluding remarks and explain how the main results can be generalized to the broader context of eventually commensurate constructions.

\section{Preliminaries}

Throughout this paper we let $\N$ be the set of all natural numbers $0, 1, 2, \dots$. Let $\N_+$ be the set of all positive integers. Let $\Z$ be the set of all integers. 

\subsection{Finite sequences, finite functions, and finite words}
Let $\mathcal{S}$ be the set of all finite sequences of natural numbers. We will introduce some operations and relations on $\mathcal{S}$. We view each element of $\mathcal{S}$ from three different perspectives, that is, as a finite sequence, as a function with a finite domain, and as a finite word. For each $s\in S$, let $\lh(s)$ denote the {\em length} of $s$. Let $()$ denote the unique (empty) sequence with length $0$. A nonempty sequence in $\mathcal{S}$ is of the form $s=(a_1, \dots, a_n)$ where $n=\lh(s)$ and $a_1, \dots, a_n\in\N$. We also view $()$ as the unique (empty) function with the empty domain, and view each nonempty $s\in \mathcal{S}$ as a function from $\{1, \dots, \lh(s)\}$ to $\N$. In addition, we refer to each $s\in \mathcal{S}$ as a {\em word} of natural numbers. When $s\in S$, the different points of view give rise to different notation for $s$; for example, we have $s=(s(1), \dots, s(\lh(s)))=s(1)\dots s(\lh(s))$.

For $s\in\mathcal{S}$ and $k\leq l\in\dom(s)$ (i.e. $1\leq k\leq l\leq \lh(s)$), define
$s\!\upharpoonright\![k,l]$ to be the unique $t\in\mathcal{S}$ with $\lh(t)=l-k+1$ such that for $1\leq i\leq \lh(t)$, $t(i)=s(k+i-1)$. Also define $s\!\upharpoonright\! k=s\!\upharpoonright\![1,k]$ and $s\!\upharpoonright\!0=()$.

For $s, t\in \mathcal{S}$, $t$ is a {\em subword} of $s$ if there are $1\leq k\leq l\leq \lh(s)$ such that $t=s\!\upharpoonright\![k,l]$. When $t$ is a subword of $s$, we also say that $t$ {\em occurs} in $s$. If $t$ is a subword of $s$ and $1\leq k\leq \lh(s)$, then we say that there is an {\em occurrence of $t$ in $s$ at position $k$} if $t=s\!\upharpoonright\![k,k+\lh(t)-1]$.
We say that $t$ is an {\em initial segment} of $s$, denoted $t\sqsubseteq s$, if $t=s\!\upharpoonright\![1,\lh(t)]$. 

For $s_1,\dots, s_n\in \mathcal{S}$, we define the {\em concatenation} $s_1^\smallfrown\dots {}^\smallfrown s_n$ to be the unique word $t\in\mathcal{S}$ with length $\sum_{j=1}^n \lh(s_j)$ such that for all $1\leq j\leq n$,
$s_j=t\!\upharpoonright\!\left[1+\sum_{i=1}^{j-1}\lh(s_i), \sum_{i=1}^j \lh(s_i)\right]$.
 
For $s\in \mathcal{S}$, define $s^0=()$ and, if $n\in\mathbb{N}_+$, define $s^n$ to be the word $s_1^\smallfrown \dots^\smallfrown s_n$ where $s_1=\dots=s_n=s$.  Words of the form $s^n$, with $\lh(s)=1$, are called {\em constant}.

For $s, t\in\mathcal{S}$ of the same length, we say that $s$ and $t$ are {\em incompatible}, denoted $s\perp t$, if $t$ is not a subword of  $s^\smallfrown(c)^\smallfrown s$ for any $c\in\N$. It is easy to check that $\perp$ is a symmetric relation for words of the same length. We say that $s$ and $t$ are {\em compatible} if it is not the case that $s\perp t$.


Let $\mathcal{F}$ be the set of all binary words that start and end with $0$. Again, each element of $\mathcal{F}$ can be equivalently viewed as a finite sequence of 0s and 1s, as a finite function with codomain $\{0,1\}$, or most often, as a finite $0,1$-word. 

\subsection{Infinite and bi-infinite sequences}
We will consider infinite sequences of natural numbers as well as infinite binary sequences. Again, they will be equivalently viewed as sequences, functions, and infinite words. We tacitly assume that an infinite sequence has domain $\N$, unless explicitly specified otherwise. 

For an infinite word $V$ and natural numbers $k\leq l$, define $V\!\upharpoonright\![k,l]$ to be the unique $s\in \mathcal{S}$ such that $\lh(s)=l-k+1$ and for all $1\leq i\leq \lh(s)$, $s(i)=V(k+i-1)$. Also define $V\!\upharpoonright\!k=V\!\upharpoonright\![0,k]$. In the same fashion as for finite words, we may speak of when a finite word $s$ is a {\em subword} of $V$ or $s$ {\em occurs} in $V$, of there being an {\em occurrence of $s$ in $V$ at position $k$} for $k\in\N$, and of $s$ being an {\em initial segment} of $V$, which is denoted as $s\sqsubseteq V$.

If $v_0\sqsubseteq v_1\sqsubseteq \dots \sqsubseteq v_n\sqsubseteq \dots$
is an infinite sequence of elements of $\mathcal{S}$ each of which is an initial segment of the next, then there is a unique infinite sequence $V$ such that $v_n\sqsubseteq V$ for all $n\in\N$. We call this unique infinite sequence the {\em limit} of $(v_n: n\in\N)$ and denote it by $\lim_n v_n$. Specifically, for each $n\in\N$ and $1\leq i\leq \lh(v_n)$, $(\lim_n v_n)(i)=v_n(i+1)$. The infinite words we consider will arise as limits of such sequences of finite words.

A {\em bi-infinite} sequence (or word) is an element of $\N^\Z$. A {\em bi-infinite binary} sequence is an element of $\{0,1\}^\Z$. The relations of {\em subword} and {\em occurrence} can be defined similarly between finite words and bi-infinite words. With $\{0,1\}$ equipped with the discrete topology and $\{0,1\}^\Z$ equipped with the product topology, $\{0,1\}^\Z$ becomes a compact metric space. The {\em shift map} $\sigma$ on $\{0,1\}^Z$ is defined as
$$ \sigma(x)(i)=x(i+1) $$
for all $x\in\{0,1\}^\Z$ and $i\in\Z$. With $\{0,1\}$ equipped with any probability measure and $\{0,1\}^\Z$ equipped with the product measure, $\sigma$ is a measure-preserving automorphism on $\{0,1\}^\Z$. 

\subsection{Symbolic rank-one systems and rank-one transformations\label{s1}}

Both symbolic rank-one systems and rank-one transformations are constructed from the so-called cutting and spacer parameters $(r_n:n\in\N)$ and $(s_n:n\in\N)$. The {\em cutting parameter} $(r_n:n\in\N)$ is an infinite sequence of natural numbers with $r_n\geq 2$ for all $n\in\N$. The {\em spacer parameter} $(s_n:n\in\N)$ is a sequence of finite sequences of natural numbers with $\lh(s_n)=r_n-1$ for all $n\in\N$.

Given cutting and spacer parameters $(r_n:n\in\N)$ and $(s_n:n\in\N)$, a symbolic  rank-one system is defined as follows. First, inductively define an infinite sequence of finite binary words $(v_n:n\in\N)$ as
$$ v_0=0,\ \ v_{n+1}=v_n1^{s_n(1)}v_n\dots v_n1^{s_n(r_n-1)}v_n. $$
We call $(v_n:n\in\N)$ a {\em generating sequence}. 
Noting that each $v_n\in \mathcal{F}$ (that is, $v_n$ starts and ends with 0) and $v_n$ is an initial segment of $v_{n+1}$, we may define
$$ V=\lim_n v_n. $$
$V$ is said to be an {\em infinite rank-one word}. 
Finally, let
$$ X=X_V=\{ x\in \{0,1\}^\Z\,:\, \mbox{every finite subword of $x$ is a subword of $V$}\}. $$
Then $X$ is a closed subspace of $\{0,1\}^\Z$ invariant under the shift map $\sigma$, i.e., $\sigma(x)\in X$ for all $x\in X$. For simplicity we still write $\sigma$ for $\sigma\!\upharpoonright\! X$. We call $(X, \sigma)$ a {\em symbolic rank-one system}. 

With cutting and spacer parameters $(r_n:n\in\N)$ and $(s_n:n\in\N)$, one can also define a rank-one measure-preserving transformation $T$ by a cutting and stacking process as follows. First, inductively define an infinite sequence of natural numbers $(h_n: n\in\N)$ as
\begin{equation}\label{h} h_0=1,\ \ h_{n+1}=r_n h_n+\sum_{i=1}^{r_n-1} s_n(i). \end{equation}
Next, define sequences $(B_n:n\in\N)$, $(B_{n, i}: n\in\N, 1\leq i\leq r_n)$ and $(C_{n,i,j}: n\in\N, 1\leq i\leq r_n-1, 1\leq j\leq s_n(i))$, all of which are subsets of $[0,+\infty)$, by induction on $n$. Define $B_0=[0,1)$ and $B_{n+1}=B_{n,1}$ for all $n\in\N$. Let $B_n$ be given and inductively assume that $T^k[B_n]$ are defined for $0\leq k<h_n$ so that $T^k[B_n]$, $0\leq k<h_n$, are all disjoint. Let $\{B_{n,i}: n\in\N, 1\leq i\leq r_n\}$ be a partition of $B_n$ into $r_n$ many sets of equal measure and let $\{C_{n,i,j}: n\in\N, 1\leq i\leq r_n-1, 1\leq j\leq s_n(i)\}$ be disjoint sets each of which is disjoint from $B_n$ and has the same measure as $B_{n,1}$. Then define $T$ so that 
for all $1\leq i\leq r_n-1$, 
$$T^{h_n}[B_{n,i}]=\left\{\begin{array}{ll} C_{n,i,1} & \mbox{if $s_n(i)>0$} \\ B_{n,i+1} & \mbox{if $s_n(i)=0$;}\end{array}\right. $$
and for $1\leq j\leq s_n(i)$,
$$ T[C_{n,i,j}]=\left\{\begin{array}{ll} C_{n,i, j+1} &\mbox{if $1\leq j<s_n(i)$} \\
B_{n,i+1} & \mbox{if $j=s_n(i)$}.\end{array}\right. $$
We have thus defined $T^k[B_{n+1}]$ for $0\leq k<h_{n+1}$ so that all of them are disjoint. Finally, let
$$ Y=\bigcup\{ T^k[B_n]: n\in\N, 0\leq k<h_n\}. $$
Then $T$ is a measure-preserving automorphism of $Y$. If $Y$ has finite Lebesgue measure, or equivalently if 
\begin{equation}\label{eqn:a} \sum_{n=0}^\infty \displaystyle\frac{h_{n+1}-h_nr_n}{h_{n+1}}<+\infty, \end{equation}
then $(Y,\lambda)$, where $\lambda$ is the normalized Lebesgue measure on $Y$, is a probability Lebesgue space. Clearly $T$ is still a measure-preserving automorphism of $(Y, \lambda)$. Such a $T$ is called a {\em rank-one (measure-preserving) transformation}.

Connecting the symbolic and the geometric constructions, we can see that $h_n=\lh(v_n)$ for all $n\in\N$. When $(\ref{eqn:a})$ holds, there is a unique probability Borel measure $\mu$ on $X$, and $(X,\mu,\sigma)$ and $(Y,\lambda, T)$ are isomorphic measure-preserving transformations. In particular, the isomorphism type of $(Y, \lambda, T)$ does not depend on the numerous choices one has to make in the process to construct $T$ (e.g. how the sets $B_{n, i}$ and $C_{n,i,j}$ are picked and how $T$ is defined on them).

Throughout the rest of the paper we tacitly assume that (\ref{eqn:a}) is satisfied for all rank-one transformations under our consideration.

For more information on the basics of rank-one transformations, particularly on the connections between the symbolic and geometric constructions, c.f. \cite{Ferenczi} \cite{GaoHill0} and \cite{GaoHill}. 

\subsection{Canonical generating sequences}

The notion of the canonical generating sequence was developed in \cite{GaoHill1} in the study of topological conjugacy of symbolic rank-one systems. In \cite{GaoHill}, however, we used the notion to characterize non-rigid rank-one transformations among all bounded  rank-one transformations. We will use this notion later in this paper again.

For $u, v\in \mathcal{F}$ we say that $u$ is {\em built from} $v$, denoted $v\prec u$, if for some $n\geq 1$ there are $a_1, \dots, a_n\in\N$ such that
$$ u=v1^{a_1}v\dots v1^{a_n}v. $$
If in addition $a_1=\dots =a_n$, the we say that $u$ is {\em simply built from} $v$, and denote $v\prec_s u$. It is easy to see that $\prec$ is a transitive relation, and $\prec_s$ is not. If $V$ is an infinite word, we also say $V$ is {\em built from} $v$, and denote $v\prec V$, if there is an infinite sequence $(a_n: n\in\N_+)$ of natural numbers such that
$$ V=v1^{a_1}v\dots v1^{a_n}v\dots. $$
Similarly, we say that $V$ is {\em simply built from} $v$ if $a_1=\dots=a_n=\dots$. We say that $V$ is {\em non-degenerate} if $V$ is not simply built from any word in $\mathcal{F}$.

Let $V$ be an infinite rank-one word. As in Subsection~\ref{s1} a generating sequence for $V$, $(v_n:n\in\N)$, is a sequence of elements of $\mathcal{F}$ such that $v_0=0$,  $v_n\prec v_{n+1}$ for all $n\in\N$, and $V=\lim_n v_n$. If follows that $v_n\prec V$ for each $n\in\N$. In general, if $v\in\mathcal{F}$ and $v\prec V$, then there is a unique way to express $V$ as an infinite concatenation of $v$ as above, and in this case the demonstrated occurrence of $v$ are called the {\em expected occurrences} of $v$.

The {\em canonical generating sequence} of $V$ is a sequence enumerating in increasing $\prec$-order the set of all $v\in\mathcal{F}$ such that there do not exist $u,w \in \mathcal{F}$ satisfying $u \prec v \prec w \prec V$ and $u \prec_s w$.  In \cite{GaoHill1} it was shown that, if $V$ is non-degenerate, the canonical generating sequence is infinite.


Throughout the rest of this paper we consider only non-degenerate infinite rank-one words. There exist cutting and spacer parameters correspondent to each generating sequence. Thus we may speak of the {\em canonical cutting and spacer parameters} given any non-degenerate infinite rank-one word. A rank-one transformation $T$ is {\em bounded} if some cutting and spacer parameters $(r_n: n\in\N)$ and $(s_n:n\in\N)$ giving rise to $T$ are bounded, i.e., there is $B>0$ such that for all $n\in\N$ and $1\leq i\leq r_n-1$, $r_n<B$ and $s_n(i)<B$. Similarly, $T$ is {\em canonically bounded} if some canonical cutting and spacer parameters giving rise to $T$ are bounded. A canonically bounded rank-one transformation is necessarily bounded, but the converse is not true. The following theorem characterizes exactly which bounded rank-one transformations are canonically bounded. 

\begin{theorem}[\cite{GaoHill}]\label{TC} Let $T$ be a bounded rank-one transformation. Then $T$ is non-rigid, i.e. $T$ has trivial centralizer, if and only if $T$ is canonically bounded.
\end{theorem}

\subsection{Replacement schemes and topological conjugacy}
Given infinite rank-one words $V$ and $W$, a {\em replacement scheme} is a pair $(v, w)$ of elements of $\mathcal{F}$, such that $v\prec V$, $w\prec W$, and for all $k\in\N$, there is an expected occurrence of $v$ in $V$ at position $k$ if and only if there is an expected occurrence of $w$ in $W$ at position $k$. This notion is closely related to the topological conjugacy between symbolic rank-one systems. 

In fact, if $v\prec V$, then every $x\in X_V$ can be uniquely expressed as
$$ x=\dots v1^{a_{-i}}v\dots v1^{a_0}v\dots v1^{a_i}v\dots $$
for $\dots a_{-i}, \dots, a_0, \dots, a_i,\dots \in \N$. We say that $x$ is {\em built from} $v$. The demonstrated occurrences of $v$ are again said to be {\em expected}. When $(v, w)$ is a replacement scheme for $V$ and $W$, we may define a map $\phi: X_V\to X_W$ so that 
$$ \phi(x)=\dots w1^{b_{-i}}w\dots w1^{b_0}w\dots w1^{b_i}w\dots $$
i.e., $\phi(x)$ is built from $w$, and so that for all $k\in\Z$, there is an expected occurrence of $v$ in $x$ at position $k$ if and only if there is an expected occurrence of $w$ in $\phi(x)$ at position $k$. Intuitively, $\phi(x)$ is obtained from $x$ by replacing every expected occurrence of $v$ in $x$ by $w$, adding or deleting 1s as necessary. It is easy to see that $\phi$ is a topological conjugacy between $X_V$ and $X_W$. We showed in \cite{GaoHill1} that all topological conjugacies essentially arise this way.

\begin{theorem}[\cite{GaoHill1}] Let $V$ and $W$ be non-degenerate infinite rank-one words. Then $(X_V, \sigma)$ and $(X_W, \sigma)$ are topologically conjugate if and only if there exists a replacement scheme for $V$ and $W$.
\end{theorem}

For the subject of this paper it is important to note that $\phi$ is also a measure-preserving isomorphism. This follows from the unique ergodicity of symbolic rank-one systems. Thus the existence of replacement schemes is a sufficient condition for two symbolic rank-one systems to be isomorphic.

We say that the two pairs of cutting and spacer parameters, $(r_n: n\in\N), (s_n:n\in\N)$ and $(q_n:n\in\N), (t_n:n\in\N)$, are {\em commensurate} if for all $n\in\N$, $r_n=q_n$ and $\sum_{i=1}^{r_n-1}\lh(s_i)=\sum_{i=1}^{r_n-1}\lh(t_i)$.

In the case of commensurate parameters, there is a straightforward way to identify replacement schemes and therefore it is easy to determine topological conjugacy.

\begin{corollary}\label{topiso} Let $(r_n:n\in\N)$ and $(s_n:n\in\N)$ be cutting and spacer parameters giving rise to non-degenerate infinite rank-one word $V$. Let $(r_n:n\in\N)$ and $(t_n:n\in\N)$ be cutting and spacer parameters giving rise to non-degenerate infinite rank-one word $W$. Suppose the two sets of parameters are commensurate. Then $(X_V,\sigma)$ and $(X_W, \sigma)$ are topologically conjugate if and only if there is $N\in\N$ such that for all $n\geq N$, $s_n=t_n$.
\end{corollary}

As mentioned above, this also gives an explicit sufficient condition for two symbolic rank-one systems to be measure-theoretically isomorphic.

\section{Isomorphism and disjointness}

\subsection{Non-isomorphism}

First we state a theorem from \cite{Hill}, Proposition 2.1, which is relevant to our results in this paper.

\begin{theorem}[\cite{Hill}]\label{mainiso}
Let $(r_n: n \in \N)$ and $(s_n: n \in \N)$ be cutting and spacer parameters giving rise to symbolic rank-one system $(X, \mu, \sigma)$. Let $(r_n: n \in \N)$ and $(t_n: n \in \N)$ be cutting and spacer parameters giving rise to symbolic rank-one system $(Y, \nu, \sigma)$. Suppose the following hold.  
\begin{enumerate}
\item[\rm (a)]The two sets of parameters are commensurate, i.e., for all $n$, $$\sum_{i=1}^{r_n-1} s_{n}(i) = \sum_{i=1}^{r_n-1} t_{n}(i).$$ 
\item[\rm (b)]  There is an $S \in \N$ such that for all $n$ and all $1\leq i \leq r_n-1$, $$s_n(i) \leq S \textnormal{ and } t_n(i) \leq S.$$
\item[\rm (c)]  There is an $R \in \N$ such that for infinitely many $n$, $$r_n \leq R \textnormal{ and } s_n \perp t_n.$$
\end{enumerate}   
Then $(X, \mu, \sigma)$ and $(Y, \nu, \sigma)$ are not isomorphic. 
\end{theorem}

\subsection{Disjointness}

\begin{theorem}\label{maindisjoint}
Let $(r_n: n \in \N)$ and $(s_n: n \in \N)$ be cutting and spacer parameters giving rise to symbolic rank-one system $(X, \mu, \sigma)$. Let $(r_n: n \in \N)$ and $(t_n: n \in \N)$ be cutting and spacer parameters giving rise to symbolic rank-one system $(Y, \nu, \sigma)$. Suppose the following hold.  
\begin{enumerate}
\item[\rm (a)] The two sets of parameters are commensurate, i.e., for all $n$, $$\sum_{i=1}^{r_n-1} s_{n}(i) = \sum_{i=1}^{r_n-1} t_{n}(i).$$ 
\item[\rm (b)]  There is an $S \in \N$ such that for all $n$ and all $1\leq i \leq r_n-1$, $$s_n(i) \leq S \textnormal{ and } t_n(i) \leq S.$$
\item[\rm (c)]  There is an $R \in \N$ such that for infinitely many $n$, $$r_n \leq R \textnormal{ and } s_n \perp t_n.$$
\item[\rm (d)]  For each $k >1$, either $(X, \mu, \sigma^k)$ or $(Y, \nu, \sigma^k)$ is ergodic.\end{enumerate}   
Then $(X, \mu, \sigma)$ and $(Y, \nu, \sigma)$ are disjoint. 
\end{theorem}

The only difference between the hypotheses of Theorems \ref{mainiso} and \ref{maindisjoint} is condition (d) above. Condition (d) is necessary for disjointness. In fact, if for some $k>1$, both $(X, \mu, \sigma^k)$ and $(Y, \nu, \sigma^k)$ are not ergodic, then they have a common factor which is a cyclic permutation on an $k$-element set, and thus the two transformations are not disjoint. It will be clear from the proof below that the theorem still holds if condition (d) is weakened to the following:
\begin{enumerate}
\item[\rm (d')] For each $1<k\leq S$, where $S$ is the bound from condition (b), either $(X, \mu,\sigma^k)$ or $(Y,\nu, \sigma^k)$ is ergodic.
\end{enumerate}

The rest of this subsection is devoted to a proof of Theorem~\ref{maindisjoint}. We will follow the approach of del Junco, Rahe, and Swanson \cite{delJuncoRaheSwanson} in their proof of minimal self-joinings for Chacon's transformation, as presented by Rudolph in his book \cite{RudolphBook}, Section 6.5.

The setup of the proof is standard. Let $\overline{\mu}$ be an ergodic joining of $\mu$ and $\nu$ on $X\times Y$. We need to show that $\overline{\mu}=\mu\times\nu$. By Lemma 6.14 of \cite{RudolphBook} (or Proposition 2 of \cite{delJuncoRaheSwanson}), it suffices to find some $k\geq 1$ such that $(X, \mu, \sigma^k)$ is ergodic and $\overline{\mu}$ is $(\sigma^k\times \id)$-invariant, where $\id$ is the identity transformation on $Y$. For this let $(x, y)\in X\times Y$ satisfy the ergodic theorem for $\overline{\mu}$, i.e., for all measurable $A\subseteq X\times Y$,
$$ \displaystyle\lim_{n\to\infty}\frac{1}{n}\sum_{i=0}^{n-1} \chi_A(\sigma^i(x), \sigma^i(y))=\overline{\mu}(A) $$
and
$$ \displaystyle\lim_{n\to\infty}\frac{1}{n}\sum_{i=0}^{n-1} \chi_A(\sigma^{-i}(x), \sigma^{-i}(y))=\overline{\mu}(A). $$
Such $(x, y)$ exists by the ergodicity of $\overline{\mu}$. Lemma 6.15 of \cite{RudolphBook} gives a sufficient condition to complete the proof. We state it below in our notation.

\begin{lemma}\label{tech} Suppose there are integers $a_n, b_n, c_n, d_n, e_n\in\Z$ for all $n\in\N$, a positive integer $k\geq 1$ and a real number $\alpha>0$ such that for all $n\in\N$,
\begin{enumerate}
\item[\rm (i)] $a_n\leq 0\leq b_n$ and $\lim_n (b_n-a_n)=+\infty$;
\item[\rm (ii)] $a_n\leq c_n\leq d_n\leq b_n$ and $a_n\leq c_n+e_n\leq d_n+e_n\leq b_n$;
\item[\rm (iii)] $d_n-c_n\geq \alpha(b_n-a_n)$; 
\item[\rm (iv)] for all $c_n\leq i\leq d_n$, $x(i)=x(i+k+e_n)$ and $y(i)=y(i+e_n)$; and
\item[\rm (v)] $(X, \mu, \sigma^k)$ is ergodic. 
\end{enumerate}
Then $\overline{\mu}$ is $(\sigma^k\times \id)$-invariant, and so $\overline{\mu}=\mu\times\nu$.
\end{lemma}

Note that Lemma~\ref{tech} has several valid variations. One variation is a symmetric version with the spaces $X$ and $Y$ switched. This version is obviously true since the setup is entirely symmetric for $X$ and $Y$. Another variation is the version in which $k\leq -1$ is a negative integer. Note that $(X, \mu, \sigma^{k})$ is ergodic if and only if $(X, \mu,\sigma^{-k})$ is ergodic. This version can be obtained by applying Lemma~\ref{tech} to $(X, \mu, \sigma^{-1})$ and $(Y, \nu, \sigma^{-1})$. Finally, we also have the variation in which both $k$ is negative and $X$ and $Y$ are switched. 

Now we claim that a slightly weaker construction already suffices: it is enough to find $a_n, b_n, c_n, d_n, e_n\in\Z$ for all $n\in\N$, a positive integer $K\geq 1$ and a real number $\alpha$ so that (i)--(iii) hold and for each $n\in\N$, (iv) holds for some nonzero $k\in\Z$ with $k\in[-K,K]$. In fact, since there are only finitely many integers between $-K$ and $K$, we get some nonzero integer $k\in[-K,K]$ and infinitely many $n$ for which the conditions (i)--(iv) of Lemma~\ref{tech} are satisfied. If $k>0$ and $(X, \mu, \sigma^k)$ is ergodic then we are done by Lemma~\ref{tech}. If $k>0$ but $(X, \mu, \sigma^k)$ is not ergodic, then by condition (d), $(Y, \nu, \sigma^k)$ is ergodic. It follows that $(Y, \nu, \sigma^{-k})$ is ergodic. Now we are done by the variation of Lemma~\ref{tech} in which both $k$ is negative and $X$ and $Y$ are switched. If $k<0$ we similarly apply other variations of the lemma. 

We now begin our construction. Let $K=S$ where $S$ is the bound in condition (b). Let 
$$ \alpha= \displaystyle\frac{1}{12(R+1)} $$
where $R$ is the bound in condition (c). Note that $R\geq 2$ because $r_n\geq 2$ for all $n\in\N$. Let $$D=\{n\in\N: r_n\leq R \mbox{ and } s_n\perp t_n\}.$$ Then $D$ is infinite by condition (c). 

Let $(v_n: n\in\N)$ be the generating sequence given by the cutting and spacer parameters $(r_n:n\in\N)$ and $(s_n:n\in\N)$. Then for each $n\in\N$, $x$ is built from $v_n$. Let $(w_n:n\in\N)$ be the generating sequence given by the cutting an spacer parameters $(r_n:n\in\N)$ and $(t_n:n\in\N)$. Then for each $n\in\N$, $y$ is built from $w_n$. By the commensurability condition (a), we have $\lh(v_n)=\lh(w_n)$ for all $n\in\N$. 

Fix an $n_0$ such that $\lh(v_{n_0})>3RS\geq 6S$. For any $n\in D$ with $n\geq n_0$, we define $$a_n, b_n, c_n, d_n, e_n\in\N$$ to satisfy (i)--(iii) and (iv) with some nonzero $k_n\in[-S, S]$. Define $$a_n=-2\lh(v_{n+1}) \mbox{ and } b_n=2\lh(v_{n+1}).$$ It is clear that (i) is satisfied. 

Before defining $c_n, d_n, e_n$ and $k_n$ we need to analyze the expected occurrences of $v_{n+1}$ in $x$ and the expected occurrences of $w_{n+1}$ in $y$. Since $y$ is built from $w_{n+1}$, by condition (b) the interval $[-S, S]$ has a nonempty intersection with some expected occurrence of $w_{n+1}$ in $y$. Fix one such expected occurrence of $w_{n+1}$ and suppose the occurrence begins at position $l$ and finishes at position $m$. Thus $a_n\leq -\lh(v_{n+1})-S\leq l\leq S$ and $m=l+\lh(v_{n+1})-1\leq S+\lh(v_{n+1})\leq b_n$.

Note that $x$ is built from $v_n$. We can then define an integer $j\in\Z$ where $|j|$ is the least such that there is an expected occurrence of $v_n$ at position $l+j$. A moment of reflection shows that $|j|\leq \frac{1}{2}(\lh(v_n)+S)$. Since $\lh(v_n)>6S$, the occurrence of $w_n$ at position $l$ and the occurrence of $v_n$ at position $l+j$ overlap for at least $\frac{1}{3}\lh(v_n)$ many positions.

Starting from the expected occurrence of $v_n$ at position $l+j$ in $x$, we examine the next $r_n$ many consecutive expected occurrences of $v_n$ in $x$. Suppose there is an occurrence of the following word in $x$ starting at position $l+j$:
$$ v_n1^{p(1)}v_n\dots 1^{p(r_n-1)}v_n $$
where $p\in \mathcal{S}$ with $\lh(p)=r_n-1$. Because $x$ is also built from $v_{n+1}$, and each expected occurrence of $v_{n+1}$ contains $r_n$ many expected occurrences of $v_n$, the above word is contained in an occurrence of $v_{n+1}1^qv_{n+1}$ for some $q\in\N$, where each demonstrated occurrence of $v_{n+1}$ is expected. Note that
$$ v_{n+1}1^qv_{n+1}=v_n1^{s_n(1)}v_n\dots 1^{s_n(r_n-1)}v_n1^qv_n1^{s_n(1)}v_n\dots 1^{s_n(r_n-1)}v_n. $$
By comparison, we get that $p$ is a subword of $s_n^\smallfrown(q)^\smallfrown s_n$.

Since $n\in D$ and therefore $s_n\perp t_n$, we conclude that $p\neq t_n$. Let $i_0$ be the least such that $1\leq i_0\leq r_n-1$ and $p(i_0)\neq t_n(i_0)$. Let
$$ h=(i_0-1)\lh(v_n)+\sum_{i=1}^{i_0-1}t_n(i). $$
Then $l+h$ is the beginning position of an expected occurrence of $w_n$ in $y$, and $l+j+h$ is the beginning position of an expected occurrence of $v_n$ in $x$. There is an occurrence of $w_n1^{t_n(i_0)}w_n$ in $y$ beginning at position $l+h$, and there is an occurrence of $v_n1^{p(i_0)}v_n$ in $x$ beginning at position $l+j+h$. 

Now we define $[c_n, d_n]$ to be the interval of overlap between the occurrence of $w_n$ in $y$ at position $l+h$ and the occurrence of $v_n$ in $x$ at position $l+j+h$. Since $|j|\leq \frac{1}{2}(\lh(v_n)+S)$ and $\lh(v_n)>6S$, we get that
$$ d_n-c_n\geq \frac{1}{3}\lh(v_n). $$
Define
$$ e_n=\lh(v_n)+t_n(i_0) $$
and
$$ k_n= p(i_0)-t_n(i_0). $$
Since $[c_n, d_n]$ is contained in the occurrence of $v_n$ at position $l+j+h$, and since $k_n+e_n=\lh(v_n)+p(i_0)$, we have that $x\!\upharpoonright\![c_n,d_n]$ and $x\!\upharpoonright\! [c_n+k_n+e_n, d_n+k_n+e_n]$ are the same words. Similarly, $[c_n, d_n]$ is also contained in the occurrence of $w_n$ at position $l+h$, and it follows that $y\!\upharpoonright\![c_n,d_n]$ and $y\!\upharpoonright\![c_n+e_n, d_n+e_n]$ are the same words. This means that (iv) is satisfied. 

Since $[c_n, d_n], [c_n+e_n, d_n+e_n]\subseteq [l,m]\subseteq[a_n, b_n]$, we know that (ii) is satisfied. Finally, 
$$\displaystyle\frac{d_n-c_n}{b_n-a_n}\geq \frac{\lh(v_n)}{3\cdot 4\lh(v_{n+1})}\geq\frac{\lh(v_n)}{12(R\lh(v_n)+RS)}\geq\alpha.$$
This shows that (iii) is satisfied. 

The proof of Theorem~\ref{maindisjoint} is complete.

\subsection{Applications to canonically bounded transformations}

Theorems~\ref{mainiso} and \ref{maindisjoint}, in combination with results in  \cite{GaoHill1}, give combinatorial criteria for isomorphism and disjointness for certain bounded rank-one transformations. These criteria in terms of the cutting and spacer parameters are, in principle, easy to check.

Let $(r_n:n\in\N)$ and $(s_n:n\in\N)$ be the cutting and spacer parameters for a rank-one transformation $T$. For integer $d>1$, consider the statement
\begin{align*}\tag{$\mbox{E}_d$} \forall N\in\N\ \exists n, i\in\N\ [\, n\geq N, 1\leq i\leq r_n-1, \mbox{and } h_N+s_n(i)\not\equiv 0 \mbox{ mod } d\,] \end{align*}
where $(h_n:n\in\N)$ is the sequence defined in equation (\ref{h}).

The following fact has been proved in \cite{GaoHill1}.

\begin{theorem}[\cite{GaoHill1}] Let $T$ be a bounded rank-one transformation with cutting and spacer parameters $(r_n:n\in\N)$ and $(s_n: n\in\N)$. Then for any integer $d>1$, $T^d$ is ergodic  if and only if {\rm ($\mbox{E}_d$)} holds.
\end{theorem}

We can now state our main result about commensurate, canonically bounded rank-one transformations.

\begin{corollary}\label{CBmain} Let $T$ be a rank-one transformation with bounded canonical cutting and spacer parameters $(r_n:n\in\N)$ and $(s_n: n\in\N)$. Let $S$ be a rank-one transformation with bounded canonical cutting and spacer parameters $(q_n:n\in\N)$ and $(t_n: n\in\N)$. Suppose the parameters for $T$ and $S$ are commensurate. Then the following hold.
\begin{enumerate}
\item $T$ and $S$ are isomorphic if and only if there is $N\in\N$ such that for all $n\geq N$, $s_n=t_n$.
\item $T$ and $S$ are disjoint if and only if for infinitely many $n\in\N$, $s_n\neq t_n$ and for every integer $d>1$, either $T^d$ is ergodic or $S^d$ is ergodic.
\end{enumerate}
\end{corollary}

As in Theorem~\ref{maindisjoint}, if $D$ is an upper bound for the sequences $(s_n:n\in\N)$ and $(t_n:n\in\N)$, then clause (2) can be strengthened to
\begin{enumerate}
\item[(2')] $T$ and $S$ are disjoint if and only if for infinitely many $n\in\N$, $s_n\neq t_n$ and for every integer $1<d\leq D$, either $T^d$ is ergodic or $S^d$ is ergodic.
\end{enumerate}

The rest of this subsection is devoted to a proof of Corollary~\ref{CBmain}. 

Let $(v_n:n\in\N)$ be the generating sequence given by the cutting and spacer parameters $(r_n:n\in\N)$ and $(s_n:n\in\N)$. Let $V=\lim_n v_n$. Then $T$ is isomorphic to the symbolic rank-one system $(X_V, \mu, \sigma)$ for a uniquely ergodic Borel probability measure $\mu$. So we will assume that $T$ is $(X, \mu, \sigma)$. Let $(w_n:n\in\N)$ be the generating sequence given by the cutting and spacer parameters $(q_n: n\in\N)$ and $(t_n:n\in \N)$. Let $W=\lim_n w_n$. We will similarly assume that $S$ is the symbolic rank-one system $(X_W, \nu, \sigma)$ for a suitable measure $\nu$. By commensurability, we have that for all $n\in\N$, $q_n=r_n$ and $\lh(v_n)=\lh(w_n)$.

First consider isomorphism. The condition is sufficient since it gives a replacement scheme, which in turn gives rise to a topological conjugacy which is also a measure-theoretic isomorphism. More specifically, if for all $n\geq N$, $s_n=t_n$, then $(v_N, w_N)$ is a replacement scheme, and the topological conjugacy it induces is an isomorphism between $T$ and $S$. 

For the necessity, assume that for infinitely many $n\in\N$, $s_n\neq t_n$. Before proceeding with the proof we prove a basic fact about compatibility.

\begin{lemma}\label{perp} Let $s, t, s', t'\in \mathcal{S}$. Suppose $s\neq t$, $\lh(s)=\lh(t)=l>0$, and $\lh(s')=\lh(t')=m>0$. Assume the following two words are compatible:
\begin{equation}\label{s*}
{s}^\smallfrown ({s'(1)})^\smallfrown {s}^\smallfrown\dots {}^\smallfrown{s}^\smallfrown (s'(m))^\smallfrown s \end{equation}
\begin{equation}\label{t*}
{t}^\smallfrown (t'(1))^\smallfrown {t}^\smallfrown\dots {}^\smallfrown{t}^\smallfrown (t'(m))^\smallfrown t. \end{equation}
Then $s'$ and $t'$ are both constant words.
\end{lemma}
\begin{proof}
Let $u$ be the word in (\ref{s*}) and $z$ be the word in (\ref{t*}). Suppose $z$ is a subword of $u^\smallfrown (c)^\smallfrown u$ for some $c\in\N$. Since $s\neq t$, the first occurrence of $t$ in $z$ cannot line up with any occurrence of $s$ in $u$, i.e., in the occurrence of $z$ in $u^\smallfrown(c)^\smallfrown u$, the starting position of the first occurrence of $t$ is not the same as the starting position of any demonstrated occurrence of $s$. Since $\lh(s)=\lh(t)=l>0$, this implies that there is $1\leq j\leq l$ such that $t'(1)=s(j)$. But then it follows that $t'(2)=\dots=t'(m)=s(j)$. Thus $t'$ is constant. By symmetry, $s'$ is also constant.
\end{proof}

Now back to the proof of Corollary~\ref{CBmain} (1). We have assumed that there are infinitely many $n\in\N$ with $s_n\neq t_n$. We inductively define an infinite sequence $(n_k: k\in\N)$ of natural numbers as follows. Define $n_0=0$. In general, assume $n_k$, $k\geq 0$, has been defined. Define $n_{k+1}=n_k+1$ if $s_{n_k}=t_{n_k}$. Otherwise, $s_{n_k}\neq t_{n_k}$, and we define $n_{k+1}=n_k+2$ if $s_{n_k+1}$ is not constant, and define $n_{k+1}=n_k+3$ otherwise. Let $v'_k=v_{n_k}$ and $w'_k=w_{n_k}$ for all $k\in\N$. Then $(v'_n:n\in\N)$ is a subsequence of $(v_n:n\in\N)$ giving rise to $T$ and $(w'_n:n\in\N)$ is a subsequence of $(w_n:n\in\N)$ giving rise to $S$. Let $(r'_n:n\in\N)$ and $(s'_n:n\in\N)$ be the cutting and spacer parameters correspondent to $(v'_n:n\in\N)$. Let $(q'_n:n\in\N)$ and $(t'_n:n\in\N)$ be the cutting and spacer parameters correspondent to $(w'_n:n\in\N)$. It is clear that the newly defined parameters are commensurate. We claim that the newly defined parameters for $T$ and $S$ satisfy all the other hypotheses of Theorem~\ref{mainiso}. Thus $T$ and $S$ are not isomorphic.

To verfity the claim, first note that $n_k< n_{k+1}\leq n_k+3$ for all $k\in\N$. This implies boundedness of the newly defined cutting and spacer parameters. In fact, if $R$ is a bound for $(r_n:n\in\N)$, then $R^3$ is a bound for $(r'_n:n\in\N)$. If $S$ is a bound for $(s_n:n\in\N)$, then $S$ is still a bound for $(s'_n:n\in\N)$. 

It remains to verify that for infinitely many $k\in\N$, $s'_k\perp t'_k$. By our construction of the sequence $(n_k:k\in\N)$, there are infinitely many $k$ such that either $n_{k+1}=n_k+2$ or $n_{k+1}=n_k+3$. We claim that for each of these $k$ we have $s'_k\perp t'_k$. First suppose $k$ is such that $n_{k+1}=n_k+2$. By our construction this means that $s_{n_k}\neq t_{n_k}$ and $s_{n_k+1}$ is not constant. In this case, we have
$$ s'_k={s_{n_k}}^\smallfrown (s_{n_k+1}(1))^\smallfrown {s_{n_k}}^\smallfrown \dots {}^\smallfrown (s_{n_k+1}(r_{n_k+1}-1))^\smallfrown s_{n_k} $$
and
$$  t'_k={t_{n_k}}^\smallfrown (t_{n_k+1}(1))^\smallfrown{t_{n_k}}^\smallfrown \dots {}^\smallfrown (t_{n_k+1}(r_{n_k+1}-1))^\smallfrown t_{n_k}.$$
By Lemma~\ref{perp}, $s'_k\perp t'_k$. Next suppose $k$ is such that $n_{k+1}=n_k+3$. By our construction this means that $s_{n_k}\neq t_{n_k}$ and $s_{n_{k+1}}$ is constant. A similar application of Lemma~\ref{perp} will complete the proof, provided that we verify the word
$$ {s_{n_k+1}}^\smallfrown (s_{n_k+2}(1))^\smallfrown {s_{n_k+1}}^\smallfrown \dots {}^\smallfrown (s_{n_k+2}(r_{n_k+2}-1))^\smallfrown s_{n_k+1} $$
is not constant. Assume it is. Note that this sequence correspond to the way $v_{n_k+3}$ is built from $v_{n_k+1}$. Thus $v_{n_k+1}\prec_s v_{n_k+3}$ and $v_{n_k+2}$ is not on the canonical generating sequence. This contradicts our assumption that $(v_n: n\in\N)$ is a canonical generating sequence.

We have thus shown Corollary~\ref{CBmain} (1). For Corollary~\ref{CBmain} (2), the necessity of the condition is clear (c.f. the remarks after the statement of Theorem~\ref{maindisjoint}). For the sufficiency, it is enough to construct new pairs of cutting and spacer parameters as above, and apply Theorem~\ref{maindisjoint}.

\section{Minimal self-joinings and Ryzhikov's theorem}

\subsection{Minimal self-joinings}
\begin{theorem}\label{mainMSJ}
Let $(r_n: n \in \N)$ and $(s_n : n \in \N)$ be cutting and spacer parameters giving rise to symbolic rank-one system $(X, \mu, \sigma)$.  Suppose the following hold. 
\begin{enumerate}
\item[\rm (a)]  For some $R$ and all $n$, $r_n \leq R$.
\item[\rm (b)]  For some $S$ and all $n$ and all $0 < i< r_n$, $s_n(i) \leq S$.
\item[\rm (c)]  For all $n$ and all $c \in \N$, there are only two occurrence of $s_n$ in ${s_n}^\smallfrown(c)^\smallfrown s_n$.
\item[\rm (d)]  $(X, \mu, \sigma)$ is totally ergodic.
\end{enumerate}
Then $(X, \mu, \sigma)$ has minimal self-joinings of all orders.
\end{theorem}

First we note a well-known fact that for rank-one transformations, having minimal self-joinings of order 2 implies minimal self-joinings of all orders. We thank Eli Glasner for providing us the references and for allowing us to include the argument here for the benefit of the reader.

\begin{theorem}\label{allorders} If a rank-one transformation has minimal self-joinings of order 2, then it has minimal self-joinings of all orders.
\end{theorem}

\begin{proof} An inductive argument (c.f. \cite{GlasnerBook} Theorem 12.16) shows that for any weakly mixing transformation, having minimal self-joinings of order 3 implies minimal self-joinings of all orders.
A theorem of Ryzhikov \cite{Ryzhikov0} states that a 2-mixing measure-preserving transformation with minimal self-joinings of order 2 has minimal self-joinings of all orders. It follows that if a transformation has minimal self-joinings of order 2 but not order 3, then it is mixing but not 2-mixing (c.f. \cite{GlasnerBook} Corollary 12.22). A theorem of Kalikow \cite{Kalikow} states that any mixing rank-one transformation is also $2$-mixing (and in fact $k$-mixing for all $k>1$). Thus one concludes that a rank-one transformation with minimal self-joinings of order 2 also has minimal self-joinings of order 3. Since having minimal self-joinings of order 2 implies weakly mixing, such a transformation has minimal self-joinings of all orders. 
\end{proof}

The above theorem is well-known to experts in the field and the references provided here are not meant to be exhaustive. For instance, the theorem was mentioned in \cite{Ryzhikov} (without proof or further references). A weaker form of the theorem was mentioned in \cite{King88}, which is sufficient for our purpose since we only consider bounded rank-one transformations, which are not mixing.

As in Theorem~\ref{maindisjoint} and Corollary~\ref{CBmain}, condition (d) of Theorem~\ref{mainMSJ} can be weakened to 
\begin{enumerate}
\item[(d')] For each $1<k\leq S$, where $S$ is the bound from condition (b), $(X, \mu, \sigma^k)$ is ergodic.
\end{enumerate}
This will be clear from the proof below.

The rest of this subsection is devoted to a proof of Theorem~\ref{mainMSJ} for minimal self-joinings of order 2. We again follow the approach of del Junco, Rahe, and Swanson \cite{delJuncoRaheSwanson} in their proof of minimal self-joinings for Chacon's transformation, as presented by Rudolph in his book \cite{RudolphBook}, Section 6.5.

Let $(v_n:n\in\N)$ be the generating sequence given by the cutting and spacer parameters $(r_n:n\in\N)$ and $(s_n:n\in\N)$. 

\begin{lemma}\label{cut} Without loss of generality, we may assume $r_n\geq 3$ for all $n\in\N$.
\end{lemma}

\begin{proof} Simply consider the subsequence $(v'_n:n\in\N)$ defined as $v'_n=v_{2n}$ for all $n\in\N$. Then $r'_n=r_{2n}r_{2n+1}\geq 4$ is the new cutting parameter, and the new spacer parameter $s'_n$ is 
\begin{equation}\label{s} {s_{2n}}^\smallfrown(s_{2n+1}(1))^\smallfrown {s_{2n}}^\smallfrown(s_{2n+1}(2))^\smallfrown\dots ^\smallfrown{s_{2n}}^\smallfrown (s_{2n+1}(r_{2n+1}-1))^\smallfrown s_{2n}. \end{equation}
If $R$ is the bound for $r_n$ in condition (a), then $r'_n\leq R^2$. If $S$ is the bound for all $s_n(i)$ in condition (b), $S$ is still a bound for all $s'_n(j)$. Since $\lim_n v_n=\lim_n v'_n$, condition (d) continues to hold. It remains only to verify that condition (c) continues to hold for $s'_n$. 

Towards a contradiction, suppose $s'_n$, which is in the form given by (\ref{s}), occurs in ${s'_n}^\smallfrown (c)^\smallfrown s'_n$ not as demonstrated. We refer to this occurrence of $s'_n$ as the {\em hidden occurrence}. Note that $s'_n$ starts with an occurrence of $s_{2n}$. Thus the hidden occurrence of $s'_n$ must start at a position where an expected occurrence of $s_{2n}$ in ${s'_n}^\smallfrown (c)^\smallfrown s'_n$ begins, because otherwise we get that $s_{2n}$ occurs in some ${s_{2n}}^\smallfrown (d)^\smallfrown s_{2n}$ not as demonstrated, contradicting our condition (c). In other words, all expected occurrence of $s_{2n}$ in the hidden occurrence of $s'_n$ must be already demonstrated in the form given by (\ref{s}). By comparison, we get that $s_{2n+1}$ occurs in ${s_{2n+1}}^\smallfrown (c)^\smallfrown s_{2n+1}$ not as demonstrated, again contradicting condition (c).
\end{proof}

For the rest of the proof we assume that $r_n\geq 3$ for all $n\in\N$.

Let $E_0$ be the set of all $x\in X$ for which there is $n\in\N$ such that the position 0 is contained in an expected occurrence of $v_n$ in $x$. Let $E=\bigcap_{k\in\Z} \sigma^k[E_0]$. Then $\mu(E)=1$. In fact, by condition (b), $X\setminus E_0$ is finite. Thus $X\setminus E$ is at most countable. 

We define a labeling function $\lambda_n: E\to \{1, \dots, r_n, \infty\}$ for each $n\in\N$. Let $n\in\N$ and $x\in E$ be given. If the position $0$ is not contained in an expected occurrence of $v_n$ in $x$, put $\lambda_n(x)=\infty$. Otherwise, the position $0$ is contained in an expected occurrence of $v_n$ in $x$, and it follows that the expected occurrence of $v_n$ (containing the position 0) is in turn contained in an expected occurrence of $v_{n+1}$ in $x$. Since there are exactly $r_n$ many expected occurrence of $v_n$ in $v_{n+1}$, we may speak of the $i$-th occurrence of $v_n$ in $v_{n+1}$ for $1\leq i\leq r_n$. Now put $\lambda_n(x)=i$ if the expected occurrence of $v_n$ containing position 0 is the $i$-th occurrence of $v_n$ in the expected occurrence of $v_{n+1}$ in $x$ containing the position 0. For any $x\in E$, $\lambda_n(x)<\infty$ for large enough $n$. We prove some basic facts about the labeling functions.

\begin{lemma}\label{orbit} If $x, y\in E$ are such that $\lambda_n(x)=\lambda_n(y)$ for all $n\geq N$ for some $N\in\N$, then $x$ and $y$ are in the same $\sigma$-orbit, i.e., there is $k\in\Z$ such that $\sigma^k(x)=y$. 
\end{lemma}

\begin{proof} We may assume without loss of generality that $\lambda_N(x)=\lambda_N(y)<\infty$. For $n\geq N$, let $l^x_n$ be the beginning position of the expected occurrence of $v_n$ in $x$ containing the position 0, and $l^y_n$ be the beginning position of the expected occurrence of $v_n$ in $y$ containing the position 0. Let $k=l^x_N-l^y_N$. Then by an easy induction on $n\geq N$ we have that for all $n\geq N$, $k=l^x_n-l^y_n$. This implies that $\sigma^k(x)=y$.
\end{proof}

\begin{lemma}\label{same} Let $x, y\in E$ and $n\in\N_+$. Suppose that  $\lambda_{n-1}(x)=\lambda_{n-1}(y)<\infty$. Let $[c,d]$ be the interval of overlap between the expected occurrence of $v_n$ in $x$ containing the position 0 and the expected occurrence of $v_n$ in $y$ containing the position 0. That is, letting $l^x_n$ be the beginning position of the expected occurrence of $v_n$ in $x$ containing the position 0 and $l^y_n$ be the beginning position of the expected occurrence of $v_n$ in $y$ containing the position 0, then $[c,d]=[l^x_n, l^x_n+\lh(v_n)]\cap [l^y_n, l^y_n+\lh(v_n)]$. Then $d-c\geq \lh(v_{n-1})$.
\end{lemma}

\begin{proof}
Suppose $\lambda_{n-1}(x)=\lambda_{n-1}(y)=i$. Then the $i$-th occurrence of $v_{n-1}$ in the expected occurrence of $v_n$ in $x$ containing $0$ has a nonempty overlap with the $i$-th occurrence of $v_{n-1}$ in the expected occurrence of $v_n$ in $y$ containing $0$. This implies that for all $1\leq j\leq r_{n-1}$, the $j$-th occurrence of $v_{n-1}$ in the expected occurrence of $v_n$ in $x$ containing $0$ has a nonempty overlap with the $j$-th occurrence of $v_{n-1}$ in the expected occurrence of $v_n$ in $y$ containing $0$. It follows that the length of $[l^x_n, l^x_n+\lh(v_n)]\setminus [c,d]$ cannot be greater than $\lh(v_{n-1})$. Since $r_{n-1}\geq 2$, we have $d-c\geq \lh(v_n)-\lh(v_{n-1})\geq \lh(v_{n-1})$.
\end{proof}

Define another labeling function $\kappa_n: E\to \{-1, 0, +1, \infty\}$ for all $n\in\N$ as follows:
$$ \kappa_n(x)=\left\{ \begin{array}{ll} -1 & \mbox{ if $\lambda_n(x)=1$}, \\
0 & \mbox{ if $2\leq \lambda_n(x)\leq r_n-1$}, \\
+1 & \mbox{ if $\lambda_n(x)=r_n$}, \\
\infty & \mbox { if $\lambda_n(x)=\infty$.}\end{array}\right. $$

\begin{lemma}\label{density} For $\mu$-a.e.~$x\in X$, the set $\{n\in\N: \kappa_n(x)=0\}$ has density at least 1/3. In particular, for $\mu$-a.e.~$x\in X$, there are infinitely many $n\in \N$ such that $\kappa_n(x)=0$.
\end{lemma}

\begin{proof}For each $N\in\N_+$ let $E_N=\{ x\in E: \kappa_N(x)<\infty\}$. Then $E_N\subseteq E_{N+1}$ for all $N\in\N_+$ and $E=\bigcup_{N\in\N_+} E_N$. For each $n\in\N_+$ and $\iota\in\{-1, 0, +1\}$, let $E_{n,\iota}=\{x\in E_n: \kappa_n(x)=\iota\}$. Then $\mu(E_{n,0})\geq \mu(E_n)/3\geq \mu(E_N)/3$ if $n\geq N$. Also, on each $E_N$ the functions $\kappa_N$, $\kappa_{N+1}$, $\dots$, are independent. By the law of large numbers, for each $N\in\N_+$ and $\mu$-a.e.~$x\in E_N$, $\{ n\geq N: \kappa_n(x)=0\}$ has density at least $1/3$. It follows that for $\mu$-a.e.~$x\in X$, $\{n\in\N: \kappa_n(x)=0\}$ has density at least $1/3$.
\end{proof}

\begin{lemma} \label{center}Let $x, y\in E$ and $n\in\N_+$. Suppose that $\kappa_{n-1}(x)=0$ and $\kappa_{n-1}(y)<\infty$. Let $[c, d]$ be the interval of overlap between the expected occurrence of $v_n$ in $x$ containing the position 0 and the expected occurrence of $v_n$ in $y$ containing the position 0. Then $d-c\geq \lh(v_{n-1})$.
\end{lemma}

\begin{proof} Suppose $\lambda_{n-1}(x)=i$. Then $1<i<r_n$. A moment of reflection gives that, in the expected occurrence of $v_n$ in $x$ containing the position 0, either the first expected occurrence of $v_{n-1}$ overlaps with the expected occurrence of $v_n$ in $y$ containing the position 0, or the last expected occurrence of $v_{n-1}$ overlaps with the expected occurrence of $v_n$ in $y$ containing the position 0. This shows that $d-c\geq \lh(v_{n-1})$.
\end{proof}

We now proceed to set up the proof for minimal self-joinings of order 2. Let $\overline{\mu}$ be an ergodic joining on $X\times X$ with marginals $\mu$. Suppose $\overline{\mu}$ is not an off-diagonal measure. We need to show that $\overline{\mu}=\mu\times\mu$. Again by Lemma 6.14 of \cite{RudolphBook} it suffices to find some nonzero $k\in\Z$ such that $\overline{\mu}$ is $(\sigma^k\times \id)$-invariant, since by our condition (d), $(X, \mu, \sigma^k)$ is ergodic. We let $(x, y)\in X\times X$ be a $\overline{\mu}$-generic pair in the sense that the following hold:
\begin{itemize}
\item $(x, y)$ satisfies the ergodic theorem for $\overline{\mu}$;
\item $x, y\in E$ are not in the same $\sigma$-orbit; and
\item the set $\{n\in\N: \kappa_n(x)=0\}$ has positive density.
\end{itemize}
Each of these properties are satisfied by $\overline{\mu}$-a.e.~pairs in $X\times X$.

As in the proof of Theorem~\ref{maindisjoint} it suffices to find $a_n, b_n, c_n, d_n, e_n, k_n\in\Z$ for all $n\in\N$, a positive integer $K\geq 1$ and a real number $\alpha>0$ so that for all $n\in\N$,
\begin{enumerate}
\item[\rm (nulla)] $0<|k_n|\leq K$;
\item[\rm (i)] $a_n\leq 0\leq b_n$ and $\lim_n (b_n-a_n)=+\infty$;
\item[\rm (ii)] $a_n\leq c_n\leq d_n\leq b_n$ and $a_n\leq c_n+e_n\leq d_n+e_n\leq b_n$;
\item[\rm (iii)] $d_n-c_n\geq \alpha(b_n-a_n)$; 
\item[\rm (iv)] for all $c_n\leq i\leq d_n$, $x(i)=x(i+k_n+e_n)$ and $y(i)=y(i+e_n)$.
\end{enumerate}
Applications of Lemma~\ref{tech} and its variations will give that $\overline{\mu}$ is $(\sigma^k\times \id)$-invariant, and so $\overline{\mu}=\mu\times\mu$.

Let $K=S$ where $S$ is the bound in condition (b). Let 
$$ \alpha=\displaystyle\frac{1}{2(R+1)^2}$$
where $R$ is the bound in condition (a). Fix an $n_0\in\N$ such that  $\lh(v_{n_0})>RS$. 
Let 
$$ D=\{ n\in\N: n> n_0, \lambda_n(x), \lambda_n(y)<\infty \mbox{ and } \lambda_n(x)\neq \lambda_n(y)\}. $$
Since $x$ and $y$ are not in the same $\sigma$-orbit, $D$ is infinite by Lemma~\ref{orbit}. 

\begin{lemma}\label{cases} There is an infinite $D'\subseteq D$ such that for all $n\in D'$, either  $\lambda_{n-1}(x)=\lambda_{n-1}(y)<\infty$, or both $\kappa_{n-1} (x)=0$ and $\kappa_{n-1}(y)<\infty$. 
\end{lemma}

\begin{proof} If $\N\setminus D$ is infinite, then 
$$ D'=\{n\in\N:n> n_0, \lambda_{n-1}(x)=\lambda_{n-1}(y)<\infty \mbox{ and } \lambda_n(x)\neq \lambda_n(y)\} $$
is infinite and $D'\subseteq D$. If $\N\setminus D$ is finite, then 
$$ D'=\{n\in\N: n> n_0, \kappa_{n-1}(x)=0,\ \kappa_{n-1}(y)<\infty \mbox{ and } \lambda_n(x)\neq\lambda_n(y)\} $$
has positive density and therefore is infinite.
\end{proof}

Fix an infinite $D'\subseteq D$ as in the above lemma. It suffices to define $a_n, b_n, c_n, d_n, e_n, k_n\in\Z$ for all $n\in D'$ as required. For the rest of the proof fix $n\in D'$. 

Let $[c, d]$ be the interval of overlap between the expected occurrence of $v_n$ in $x$ containing the position 0 and the expected occurrence of $v_n$ in $y$ containing the position $0$.  By Lemmas~\ref{cases}, \ref{same} and \ref{center}, we have that 
$$ d-c\geq \lh(v_{n-1})\geq \frac{\lh(v_{n-1})}{R\lh(v_{n-1})+RS}\lh(v_n)\geq  \frac{1}{R+1}\lh(v_n). $$

Define 
$$ a_n=-\lh(v_{n+1}) \mbox{ and } b_n=\lh(v_{n+1}). $$
Let $l=l^y_{n+1}$. Then the expected occurrence of $v_{n+1}$ in $y$ containing the position 0 starts at the position $l$. Suppose this occurrence finishes at position $m$. Then $a_n\leq l\leq 0\leq m\leq b_n$. 

Let $i_y=\lambda_n(y)$. Then in $y$, the position 0 is contained in the $i_y$-th occurrence of $v_n$ in the expected occurrence of $v_{n+1}$ from position $l$ to position $m$. Correspondingly in $x$, we examine the $r_n$ many consecutive expected occurrences of $v_n$ so that the position 0 is contained in the $i_y$-th occurrence of $v_n$. Suppose the following word is observed:
$$ v_n1^{p(1)}v_n1^{p(2)}\dots v_n1^{p(r_n-1)}v_n. $$
Since $\lambda_n(x)\neq\lambda_n(y)$, this observed word is not contained in a single expected occurrence of $v_{n+1}$. Rather, it is contained in a subword of $x$ of the form $v_{n+1}1^qv_{n+1}$, where each demonstrated occurrence of $v_{n+1}$ is expected. By comparison, we obtain that $p$ is a subword of $s_n^\smallfrown (q)^\smallfrown s_n$, and that $p$ does not coincide with any of the two demonstrated occurrences of $s_n$. By our condition (c), this implies that $p\neq s_n$. 

Let $i_0$ be such that $1\leq i_0\leq r_n-1$ and $p(i_0)\neq s_n(i_0)$ and so that $|i_0-i_y|$ is the least. For definiteness first assume that $i_0\geq i_y$. In this case let 
$$ h=(i_0-i_y)\lh(v_n)+\sum_{i=i_y}^{i_0-1}s_n(i). $$
Then in $x$ there is an occurrence of the word $v_n1^{p(i_0)}v_n$ beginning at the position $l^x_n+h$. Similarly, in $y$ there is an occurrence of the word $v_n1^{s_n(i_0)}v_n$ beginning at the position $l^y_n+h$. Define $[c_n, d_n]$ to be the interval of overlap between the these first demonstrated occurrences of $v_n$ in $x$ and in $y$. Then we have in fact $c_n=c+h$ and $d_n=d+h$. So
$$ d_n-c_n=d-c\geq \frac{1}{R+1}\lh(v_n). $$
Define 
$$e_n=\lh(v_n)+s_n(i_0)$$
and
$$ k_n=p(i_0)-s_n(i_0). $$
We have that $x\!\upharpoonright\![c_n,d_n]=x\!\upharpoonright [c_n+k_n+e_n, d_n+k_n+e_n]$ and $y\!\upharpoonright\![c_n,d_n]=y\!\upharpoonright\![c_n+e_n, d_n+e_n]$. Since $[c_n, d_n], [c_n+e_n, d_n+e_n]\subseteq [l,m]\subseteq [a_n, b_n]$ and 
$$ \displaystyle\frac{d_n-c_n}{b_n-a_n}\geq \frac{\lh(v_n)}{(R+1)\cdot 2\lh(v_{n+1})}\geq \frac{1}{2(R+1)^2}=\alpha, $$
our proof is complete in this case.

The alternative is the case $i_0<i_y$. In this case we let instead
$$ h=(i_0-i_y+1)\lh(v_n)-\sum_{i=i_0+1}^{i_y-1}s_n(i)\leq 0. $$
Then in $x$ there is an occurrence of the word $v_n1^{p(i_0)}v_n$ where the beginning of the second demonstrated occurrence is at the position $l^x_n+h$. Similarly, in $y$ there is an occurrence of the word $v_n1^{s_n(i_0)}v_n$ where the beginning of the second demonstrated occurrence is at the position $l^y_n+h$. We similarly let $[c_n, d_n]$ be the interval of overlap of these second occurrences of $v_n$ in $x$ and in $y$. Then $c_n=c+h$ and $d_n=d+h$. Define
$$ e_n=-\lh(v_n)-s_n(i_0) $$
and
$$ k_n=-p(i_0)+s_n(i_0). $$
We still have that $d_n-c_n\geq \lh(v_n)/(R+1)$, and the proof is similarly completed.

We have thus shown that $(X, \mu, \sigma)$ has minimal self-joinings of order 2, and therefore minimal self-joinings of all orders.

\subsection{Ryzhikov's theorem}

As a corollary to Theorem~\ref{mainMSJ}, we obtain the following theorem of Ryzhikov \cite{Ryzhikov} on minimal self-joinings for non-rigid, totally ergodic, bounded rank-one transformations.

\begin{corollary}[Ryzhikov \cite{Ryzhikov}]\label{Ryzh} Let $T$ be a bounded rank-one transformation. Then $T$ has minimal self-joinings of all orders if and only if $T$ is non-rigid and totally ergodic.
\end{corollary}

It is easy to verify that having minimal self-joinings implies mild mixing (having no rigid factors), which implies non-rigidity. Having minimal self-joinings also implies weak mixing, which implies total ergodicity. Thus the two conditions are necessary.

For the sufficiency, let $T$ be a bounded rank-one transformation with cutting and spacer parameters $(r_n:n\in\N)$ and $(s_n: n\in\N)$. Assume that $T$ is non-rigid and totally ergodic. By Theorem~\ref{TC}, $T$ is canonically bounded. Thus we may assume without loss of generality that $(r_n:n\in\N)$ and $(s_n:n\in\N)$ are canonical cutting and spacer parameters, which are also bounded. Let $(v_n: n\in\N)$ be the canonical generating sequence given by $(r_n:n\in\N)$ and $(s_n:n\in\N)$. We inductively define an infinite sequence $(n_k:k\in\N)$ of natural numbers as follows. Define $n_0=0$. In general, assume $n_k$, $k\geq 0$, has been defined. Define $n_{k+1}=n_k+2$ if $s_{n_k+1}$ is not constant, and define $n_{k+1}=n_k+3$ otherwise. Let $v_k'=v_{n_k}$ for all $k\in\N$. Then $(v'_n:n\in\N)$ is a subsequence of $(v_n:n\in\N)$, which still generates $T$. Let $(r'_n:n\in\N)$ and $(s'_n:n\in\N)$ be the cutting and spacer parameters corresponding to $(v'_n:n\in\N)$. Since $n_k<n_{k+1}\leq n_k+3$ for all $k\in\N$, these newly defined cutting and spacer parameters are still bounded.

To prove the corollary, we will apply Theorem~\ref{mainMSJ} to $(r'_n:n\in\N)$ and $(s'_n:n\in\N)$. The only condition to verify is (c), that is, for all $n\in\N$ and $c\in\N$, there are only two occurrences of $s'_n$ in ${s'_n}^\smallfrown(c)^\smallfrown s'_n$. Note that for every $k>0$, $s'_k$ is of the form
$$ {s_{n_k}}^\smallfrown (u(1))^\smallfrown {s_{n_k}}^\smallfrown\cdots {}^\smallfrown (u(m))^\smallfrown{s_{n_k}} $$
where $u$ is either $s_{n_k+1}$ or 
$$ {s_{n_k+1}}^\smallfrown (s_{n_k+2}(1))^\smallfrown {s_{n_k+1}}^\smallfrown\cdots {}^\smallfrown(s_{n_k+2}(r_{n_k+2}-1))^\smallfrown {s_{n_k+1}}. $$
As in the proof of Corollary \ref{CBmain}, $u$ is not constant in either cases: in the former case $s_{n_k+1}$ is assumed not to be constant, and in the latter case $u$ corresponds to the way $v_{n_k+2}$ is built from $v_{n_k}$, and therefore is not constant since $v_{n_k+1}$ is assumed to be on the canonical generating sequence. Now if there is $c\in\N$ so that $s'_n$ occurs in ${s'_n}^\smallfrown (c)^\smallfrown {s'_n}$ not as demonstrated, then by a similar argument as the proof of Lemma~\ref{perp}, it would follow that $u$ is constant, a contradiction.

This completes the proof of Corollary~\ref{Ryzh}.

\section{Concluding remarks}

Some results of this paper are applicable in a broader context than stated. We have noted that Theorems \ref{maindisjoint}, \ref{mainMSJ} and Corollary \ref{CBmain} can be strengthened with ``partial total ergodicity'' assumptions replacing the total ergodicity assumptions, which we denoted by (d') and (2') respectively. Here we note that Theorems \ref{mainiso}, \ref{maindisjoint} and Corollary \ref{CBmain} can be further strengthened with an ``eventual commensurability'' assumption replacing the commensurability assumption. For instance, Theorem~\ref{mainiso} can be strengthened as follows.

\begin{theorem} \label{mainisogen}
Let $(r_n: n \in \N)$ and $(s_n: n \in \N)$ be cutting and spacer parameters giving rise to symbolic rank-one system $(X, \mu, \sigma)$. Let $(v_n:n\in\N)$ be the generating sequence given by $(r_n:n\in\N)$ and $(s_n:n\in\N)$. 

Let $(q_n: n \in \N)$ and $(t_n: n \in \N)$ be cutting and spacer parameters giving rise to symbolic rank-one system $(Y, \nu, \sigma)$. Let $(w_n:n\in\N)$ be the generating sequence given by $(q_n:n\in\N)$ and $(t_n:n\in\N)$. 

Suppose the following hold.  
\begin{enumerate}
\item[\rm (a)]The two sets of parameters are ``eventually commensurate'', i.e., there are $N, M\in\N$ such that $\lh(v_N)=\lh(w_M)$ and for all $n\in\N$, $r_{N+n}=q_{M+n}$ and
$$\sum_{i=1}^{r_{N+n}-1} s_{N+n}(i) = \sum_{i=1}^{q_{M+n}-1} t_{M+n}(i).$$ 
\item[\rm (b)]  There is an $S \in \N$ such that for all $n$ and all $1\leq i \leq r_n-1$, $$s_n(i) \leq S \textnormal{ and } t_n(i) \leq S.$$
\item[\rm (c)]  There is an $R \in \N$ such that for infinitely many $n$, $$r_n \leq R \textnormal{ and } s_n \perp t_n.$$
\end{enumerate}   
Then $(X, \mu, \sigma)$ and $(Y, \nu, \sigma)$ are not isomorphic. 
\end{theorem}

Theorem~\ref{maindisjoint} and Corollary~\ref{CBmain} allow similar generalizations. It should be clear that the proofs of these generalizations are identical to the proofs given in \cite{Hill} and this paper. 

It is, however, not clear how to determine if two rank-one transformations allow eventually commensurate cutting and spacer parameters. Of course, if two rank-one transformations do not allow eventually commensurate parameters, then they are not isomorphic. We conjecture that there is a Borel procedure for this determination.

\section*{Acknowledgments}

The first author acknowledges the US NSF grant DMS-1201290 for the support of his research. He also acknowledges the support of the Issac Newton Institute (INI) for Mathematical Sciences at the University of Cambridge for a research visit during which a substantial part of this paper was written. He was a Visiting Fellow to the Mathematical, Foundational and Computational Aspects of the Higher Infinite (HIF) program at the INI, and he thanks the organizers of the program and the Scientific Advisory Committee for this opportunity. Both authors would like to thank Eli Glasner for useful discussions on the topics of the paper and for providing the references related to Theorem~\ref{allorders}. Both authors also benefit from discussions with Matt Foreman, Cesar Silva, and Benjy Weiss as a part of a SQuaRE program at the American Institute of Mathematics (AIM) focusing on the isomorphism problem of rank-one transformations.

\end{document}